\documentclass[reqno,centertags, 12pt, draft]{article}

\usepackage{amsmath,amsthm,amscd,amssymb}

\usepackage{latexsym}

\usepackage{mathrsfs}

\newcommand{\bbN}{{\mathbb{N}}}

\newcommand{\bbR}{{\mathbb{R}}}

\newcommand{\bbC}{{\mathbb{C}}}

\newcommand{\no}{\nonumber}

\newcommand{\ti}{\tilde  }

\newcommand{\beq}{\begin{equation}}

\newcommand{\eeq}{\end{equation}}

\newcommand{\ba}{\begin{align}}

\newcommand{\ea}{\end{align}}

\newcommand{\vphi}{\varphi}

\allowdisplaybreaks

\numberwithin{equation}{section}

\newtheorem{theorem}{Theorem}[section]

\newtheorem{lemma}[theorem]{Lemma}

\newtheorem{corollary}[theorem]{Corollary}

\theoremstyle{definition}

\theoremstyle{remark}

\newtheorem*{remark*}{Remark}

\title{Sine kernel asymptotics for a class of singular measures}
\author{Jonathan Breuer\\
\footnotesize Institute of Mathematics,  The Hebrew University of Jerusalem, Givat Ram, \\
\footnotesize 91904 Jerusalem, Israel. \\
\footnotesize Email: jbreuer@math.huji.ac.il}
\date{}
\begin{document}
\maketitle
\begin{abstract}
We construct a family of measures on $\bbR$ that are purely singular with respect to Lebesgue measure, and yet exhibit universal sine-kernel asymptotics in the bulk. The measures are best described via their Jacobi recursion coefficients: these are sparse perturbations of the recursion coefficients corresponding to Chebyshev polynomials of the second kind. We prove convergence of the renormalized Christoffel-Darboux kernel to the sine kernel for any sufficiently sparse decaying perturbation. 
\end{abstract}
\sloppy
\section{Introduction}
Let $d\mu(x)=w(x)dx+d\mu_{\textrm{sing}}(x)$ be a compactly supported measure on $\bbR$, (where $d\mu_{\textrm{sing}}$ is the part of $d\mu$ that is 
singular with respect to Lebesgue measure), and let $\{p_n\}_{n=0}^\infty$ be the sequence of orthogonal polynomials 
associated with $\mu$. Namely,
\beq \no
p_n=\gamma_n x^n+\textrm{lower order}
\eeq
with $\gamma_n>0$ and
\beq \no
\int_\bbR p_n(x)p_m(x) d\mu(x)=\delta_{m,n}.
\eeq
The Christoffel-Darboux (CD) kernel,
\beq \no
K_n(x,y)=\sum_{j=0}^{n-1}p_j(x)p_j(y),
\eeq
is the kernel of the projection onto the subspace of $L^2(d\mu)$ of polynomials of degree less than $n$. 
It arises in various natural contexts and its properties have been the focus of many works (for reviews see \cite{NevFr, Simon-CD}). 
A significant portion of these works study asymptotics of $K_n(x+\frac{a}{n},x+\frac{b}{n})$ as $n \rightarrow \infty$. 
There are two main motivations for studying these asymptotics. First, the CD kernel arises as the correlation kernel for the eigenvalues of the 
unitary ensembles of Hermitian matrices and so, its asymptotics describe the asymptotic distribution for these eigenvalues (see, e.g., \cite{deift}). 
Second, if $y$ is a zero of $p_n$ then $x$ is a zero of $p_n$ iff $K_n(x,y)=0$. 
Thus, the asymptotic properties of $K_n(x+\frac{a}{n},x+\frac{b}{n})$ are connected to the small scale behavior
of the zeros of the $p_n$ around $x$ as $n \rightarrow \infty$ (see, e.g., \cite{Freud, LeLu, Simon-CD}). 

Until recently, except for some classical cases, where the asymptotics of the $p_n$ were well investigated, the general methods for studying 
\mbox{$K_n(x+\frac{a}{n},x+\frac{b}{n})$} required $d\mu_{\textrm{sing}}=0$ and some degree of smoothness from $w(x)$ (see, e.g., \cite{lubinskyRev} and references therein). In these cases,
it was shown that, for $x_0$ in the interior of the support of the measure,
\beq \label{1.1}
\lim_{n \rightarrow \infty}\frac{K_n(x_0+\frac{a}{n},x_0+\frac{b}{n})}{K_n(x_0,x_0)}=\frac{\sin(\pi \rho(x_0) (b-a))}{\pi \rho(x_0) (b-a)}
\eeq 
where $\rho(x)$ is the asymptotic density of the zeros of $p_n$ at $x$. 
The limit in \eqref{1.1} is known as the {\it universality limit in the bulk} since, apart from the normalizing factor of $\rho(x_0)$, the limiting kernel is 
independent of $x_0$ and the particular form of $\mu$.
Two new methods introduced by Lubinsky \cite{lubinsky1, lubinsky2} enable the derivation of such a limit under much weaker requirements from the measure. 
In particular, in \cite{lubinsky1} it was shown that if $\mu$ is a regular measure on $(-2,2)$ which is absolutely continuous on a 
neighborhood of $x_0 \in \textrm{ interior of supp}(\mu)$ and has a continuous and positive Radon-Nikodym derivative at $x_0$ then \eqref{1.1} holds 
uniformly for $a,b$ in compact subsets of the complex plane. It is important to note that continuity of $w$ at $x_0$ can be replaced by a Lebesgue point 
type condition. Moreover, there exist some extensions of this result to more general sets and less restrictive conditions on the derivative of the measure (see \cite{als, findley, Simon-ext, Simon-CD, totik}). However, to the best of our knowledge, all existing methods for obtaining \eqref{1.1} require absolute continuity of the measure.

The purpose of this note is to present a class of purely singular measures for which \eqref{1.1} holds. 
We shall construct these measures through their Jacobi parameters---the parameters entering in the recursion relation of the $p_n$'s:
\beq \label{1.2}
x p_n(x)=a_{n+1}p_{n+1}(x)+b_{n+1}p_n(x)+a_n p_{n-1}(x) \quad n>0, \\
\eeq
\beq \label{1.2.1}
x p_0(x)=a_1 p_1(x)+b_1 p_0(x).
\eeq
It is a classical result that such a relation is satisfied by the set of orthogonal polynomials associated with any compactly supported, infinitely 
supported measure, with $a_n>0$ and $b_n \in \bbR$ both bounded sequences (by `infinitely supported' we mean that the support is not a finite set). 
On the other hand, any Jacobi matrix, 
\beq \label{1.3}
J \left(\{a_n,b_n\}_{n=1}^\infty \right)=\left(
\begin{array}{ccccc}
b_1    & a_1 & 0      & 0      & \dots \\
a_1    & b_2 & a_2    & 0      & \ddots \\
0      & a_2 & b_3    & a_3    & \ddots \\
\vdots & \ddots   & \ddots & \ddots & \ddots \\
\end{array} \right)
\eeq
(with the $a_n$'s positive and bounded and $b_n$'s bounded), can be viewed as a bounded self-adjoint operator in $\ell^2$ with $(1,0,0,0,\ldots)^T$ a cyclic vector. 
Thus, by the spectral theorem, $J$ and $(1,0,0,0,\ldots)^T$ have a spectral measure associated with them. The mappings $J \mapsto \mu$ via the spectral theorem and 
$\mu \mapsto J$ via the orthogonal polynomial recursion relation, for bounded Jacobi matrices and compactly supported, infinitely supported 
probability measures, can be shown to be inverses of each other (see e.g.\ \cite{deift}), and so we obtain a $1-1$ correspondence between these two families of objects. 

Perhaps the simplest case is that of the (rescaled) Chebyshev polynomials of the second kind.
In this case, $d\mu_0(x)=\frac{\sqrt{4-x^2}}{2 \pi} \chi_{[-2,2]}(x)dx$, and 
the orthogonal polynomials (for $x=2\cos (\theta)$) are $U_n(x)=\frac{\sin((n+1)\theta)}{\sin(\theta)}$. 
The asymptotic density of zeros is $\rho_0(x)=\pi^{-1} (\sqrt{4-x^2})^{-1} \chi_{[-2,2]}(x)$ and the corresponding Jacobi matrix is
\beq \label{1.4}
J_0=\left(
\begin{array}{ccccc}
0    & 1 & 0      & 0      & \dots \\
1    & 0 & 1    & 0      & \ddots \\
0      & 1 & 0    & 1    & \ddots \\
\vdots & \ddots   & \ddots & \ddots & \ddots \\
\end{array} \right).
\eeq

We shall obtain our family of measures by adding a decaying sparse perturbation to the Jacobi matrix $J_0$. 
That is, we shall consider the 
Jacobi parameters 
\beq \label{1.10}
a_n \equiv 1 \qquad b_n= \left \{ \begin{array}{ccc} v_j  &n=N_j \\
0  &\textrm{otherwise} \end{array} \right.
\eeq
where $\{N_j \}_{j=1}^\infty$ is an increasing sequence of natural numbers satisfying 
\beq \label{cond1}
\frac{N_{j+1}}{N_j} \rightarrow \infty
\eeq 
and 
\beq \label{cond2}
v_j \rightarrow 0.
\eeq 
Such matrices are known as sparse Jacobi matrices and have 
served as the first explicit examples of discrete Schr\"odinger operators with 
singular continuous spectral measures. The review \cite{last-review} contains a survey of some of the extensive research carried out in this context since Pearson's paper \cite{pearson}, where Schr\"odinger operators  with sparse decaying potentials were introduced. Here we shall rely on Theorem 1.7 in \cite{kls}, which says that when \eqref{cond1} and \eqref{cond2} hold then the spectral measure
is absolutely continuous on $(-2,2)$ if $\sum_{j=1}^\infty v_j^2 < \infty$, and purely singular continuous there if $\sum_{j=1}^\infty v_j^2 = \infty$. 

We shall prove the following
\begin{theorem} \label{theorem}
Let $\{v_j\}_{j=1}^\infty$ be a sequence of real numbers such that $v_j \rightarrow 0$ as $j \rightarrow \infty$. 
If the sequence $\{N_j\}_{j=1}^\infty$ is sufficiently sparse $($see below$)$ and $\mu$ is the measure corresponding to the 
Jacobi parameters given by \eqref{1.10}, then for every $x \in (-2,2)$ and any $a, b \in \bbR$
\beq \label{1.12}
\lim_{n \rightarrow \infty} \frac{ K_n(x+\frac{a}{n},x+\frac{b}{n})}{K_n (x,x)}=
\frac{\sin \left((\sqrt{4-x^2})^{-1} (b-a) \right)}
{(\sqrt{4-x^2})^{-1} (b-a)},
\eeq
where $K_n(x,y)$ is the corresponding CD kernel.
\end{theorem}

\begin{remark*}
By `$\{N_j\}_{j=1}^\infty$ is sufficiently sparse' we mean that $N_{k+1}$ has to be chosen sufficiently large as a function of $\{N_1,N_2,\ldots, N_k \}$. In other words, for any $k \geq 1$ there exists a function $\widetilde{N}_k(N_1, N_2, \ldots, N_k)$ such that $N_{k+1} \geq \widetilde{N}_k(N_1, N_2, \ldots N_k)$. The sequence of functions $\widetilde{N}_k$ depends on $\{v_j\}_{j=1}^\infty$.
\end{remark*}

\begin{corollary} \label{corollary1}
There exist purely singular measures such that \eqref{1.12} holds for every $x \in (-2,2)$.
\end{corollary}
\begin{proof}
As remarked above, Theorem 1.7 in \cite{kls} says that if $\sum_{j=1}^\infty v_j^2 =\infty$ and $\frac{N_{j+1}}{N_j}\rightarrow \infty$ then the measure is purely singular. Thus, by picking such sequences that satisfy the hypothesis of Theorem \ref{theorem}, we get a purely singular measure satisfying \eqref{1.12}.
\end{proof}

The idea of the proof of Theorem \ref{theorem} is quite simple. For any finite rank perturbation of $J_0$, universality holds. Thus, having chosen 
$\{N_j\}_{j=1}^K$, one may place $N_{K+1}$ only after the renormalized CD kernel is very close to its sine kernel limit. The heart of the proof lies in showing that, since $v_{K+1}$ is small, the perturbation at $N_{K+1}$ is weak and does not produce a substantial change in the renormalized CD kernel. 

After obtaining some preliminary results in Section 2, we prove Theorem \ref{theorem} in Section 3.

\textbf{Acknowledgments} We thank Yoram Last and Barry Simon for useful discussions. This research was supported by The Israel Science Foundation (Grant no.\ 1105/10).


\section{Preliminaries}

As our analysis is a perturbative analysis, we begin by sketching a proof of universality for the unperturbed model, namely, the second kind Chebyshev polynomials which correspond to the matrix $J_0$. 
A useful device, which will also play a central role in this paper, is the Christoffel-Darboux formula:
\beq \label{1.5}
K_n(x,y)=a_n \frac{p_n(x)p_{n-1}(y)-p_n(y)p_{n-1}(x)}{x-y},
\eeq 
\beq \label{1.6}
K_n(x,x)=a_n \left(p_n'(x)p_{n-1}(x)-p_n(x)p_{n-1}'(x)\right).
\eeq
Using this formula, one can show directly that, for $d\mu(x)=\frac{\sqrt{4-x^2}}{2 \pi}\chi_{[-2,2]}(x)dx$, $x \in (-2,2)$, and $a,b \in \bbC$, with $a \neq b$
\beq \label{1.7}
\begin{split}
& \lim_{n \rightarrow \infty} \frac{K_n(x+\frac{a}{n},x+\frac{b}{n})}{n}= \\
& \lim_{n \rightarrow \infty} \frac{U_n(x+\frac{a}{n})U_{n-1}(x+\frac{b}{n})-
U_n(x+\frac{b}{n})U_{n-1}(x+\frac{a}{n})}{b-a} =\\
& \frac{2\sin \left((\sqrt{4-x^2})^{-1} (b-a) \right)}{\sqrt{4-x^2} (b-a)}
\end{split}
\eeq
and the convergence is uniform in $|a|,|b|<C$ and $|b-a|> \delta$ for every $\delta, C >0$. In fact, it is not hard to see directly that the restriction $|a-b|>\delta$ is unnecessary, but we want to use an argument which will play an important role in what follows.
Note that for fixed $a \in \bbC$, $K_n(x+\frac{a}{n},x+\frac{b}{n})$ is analytic as a function of $b$. The limit function in \eqref{1.7} is analytic as well. By the uniform convergence in each annulus around $a$, it follows from Cauchy's integral formula that convergence holds also for $b =a$ and in fact is uniform in $|a|,|b|<C$ (where we interpret $\frac{\sin(0)}{0}=1$).

By considering the limit for $a=b=0$, this immediately implies 
\beq \label{1.9}
\begin{split}
\lim_{n \rightarrow \infty} \frac{K_n(x+\frac{a}{n},x+\frac{b}{n})}{K_n(x,x)}&=\lim_{n \rightarrow \infty}  \frac{K_n(x+\frac{a}{n},x+\frac{b}{n})}{n} \lim_{n \rightarrow \infty} \frac{n}{K_n(x,x)} \\ &=
\frac{\sin \left((\sqrt{4-x^2})^{-1} (b-a) \right)}
{(\sqrt{4-x^2})^{-1} (b-a)}
\end{split}
\eeq
which is precisely \eqref{1.1} (recall $\rho_0(x)=\pi^{-1} (\sqrt{4-x^2})^{-1} \chi_{[-2,2]}(x)$).

Now fix sequences $\{v_j\}_{j=1}^\infty$ and $\{N_j\}_{j=1}^\infty$ with $N_{j+1}/N_j \rightarrow \infty$, and let $\mu$ be the spectral measure corresponding to the Jacobi parameters given by \eqref{1.10}.
Let $\{p_n(x) \}_{n=0}^\infty$ be the orthogonal polynomials associated with $\mu$. 

We use variation of parameters. We consider \eqref{1.10} as a perturbation on $J_0$. Fix $x \in (-2,2)$ and let $\psi^1_n(x)$ and $\psi^2_n(x)$ be the two solutions of the difference equation
$x \psi_n(x)=\psi_{n+1}(x)+ \psi_{n-1}(x)$ for $n \geq 0$, satisfying the boundary conditions 
\beq \label{2.2}
\psi^1_0(x)=1,\ \psi^1_{-1}(x)=0 \quad \psi^2_0(x)=0,\ \psi^2_{-1}(x)=1. 
\eeq
Explicitly, it is easy to see that if $x=2 \cos (\theta)$ for $\theta \in (0,\pi)$, then 
\beq \label{2.3} 
\psi^1_{n}(x)=\frac{\sin(n+1)\theta}{\sin(\theta)}
\eeq 
(the second kind Chebyshev polynomials) and 
\beq \label{2.4}
\psi^2_n(x)=\frac{-\sin(n \theta)}{\sin(\theta)}.
\eeq

Now, define $A_n(x) \in \bbC^2$ ($n \geq 1$) by
\beq \label{2.5}
\begin{split}
p_n(x)&=A_{n,1}(x)\psi^1_n(x)+A_{n,2}(x) \psi^2_n(x) \\
p_{n-1}(x)&=A_{n,1}(x)\psi^1_{n-1}(x)+A_{n,2}(x) \psi^2_{n-1}(x),
\end{split}
\eeq
or, in matrix form,
\beq \label{2.6}
\left(\begin{array}{c} p_n(x) \\ p_{n-1}(x) \end{array} \right)=\left( \begin{array}{cc} \psi^1_n(x)  &  \psi^2_n(x) \\ 
\psi^1_{n-1}(x) & \psi^2_{n-1}(x) \end{array} \right) \left(\begin{array}{c} A_{n,1}(x) \\ A_{n,2}(x) \end{array} \right).
\eeq
We denote 
\beq \label{2.7}
T_n(x)=\left( \begin{array}{cc} \psi^1_n(x)  &  \psi^2_n(x) \\ 
\psi^1_{n-1}(x) & \psi^2_{n-1}(x) \end{array} \right)
=\left( \begin{array}{cc} x  &  -1 \\ 
1 & 0 \end{array} \right)^{n}
\eeq
and note that $\det T_n(x)=1$. Moreover, for any closed interval $I \subseteq (-2,2)$ there exists $M_I>0$ such that for any $n$ and any $x \in I$, 
\beq \label{2.8}
\parallel T_n(x) \parallel \leq M_I.
\eeq
In fact, it will be crucial later on, to be able to extend this bound slightly to the complex plane:
\begin{lemma} \label{bound-extension}
Let $I \subset (-2,2)$ be a closed interval. There exists $M_I>0$ such that for any $x \in I$, $t \in \bbR$ with $|t| \leq 1$,
\beq \label{2.8.1}
\left \| T_n \left(x+\frac{it}{n}\right) \right \| \leq M_I.
\eeq
\end{lemma}
\begin{proof}
We shall show that we can uniformly bound $\left | \psi_n^{1,2} \left (x+\frac{it}{n} \right) \right|$.
Fix $x, t$ and let $2\cos (\theta_0) =x$ and $2 \cos (\theta_0+\delta_n)=x+\frac{it}{n}$. By expanding to a Taylor series
\beq \no
\frac{it}{n}=2 \cos(\theta_0+\delta_n)-2 \cos(\theta_0)=2\delta_n \sin(\theta_0)+o(\delta_n)
\eeq
as $\delta_n \rightarrow 0$ (and so as $n \rightarrow \infty$),
we see that $\delta_n=O(\frac{1}{n})$ and the implicit constant depends on $\left |\sin (\theta_0) \right|^{-1}$ which is uniformly bounded on $I$. Write
\beq \no
\frac{\sin(n (\theta_0+\delta_n))}{\sin(\theta_0+\delta_n)}=\frac{\sin(n\theta_0)\cos(n \delta_n)+\cos(n\theta_0) \sin(n\delta_n)}{\sin(\theta_0+\delta_n)}.
\eeq
The denominator above is bounded from below on $I$, $\sin(n \theta_0)$ and $\cos(n \theta_0)$ are both uniformly bounded on $I$ and $n \delta_n$ is uniformly bounded on $I$ as well, by the discussion above. Therefore, $\psi_n^{1,2}\left(x+\frac{it}{n}\right)$ are uniformly bounded on $I$ and we are done.
\end{proof}
We assume, without loss of generality, that $M_I \geq 1$ for all I. For concreteness, we let $I_j=[-2+\frac{1}{j},2-\frac{1}{j}]$ for $j \geq 1$ and $M_j \equiv M_{I_j}$.

It is well known (and follows from \eqref{2.5}/\eqref{2.6}) that $A_n(x)$ satisfies the recurrence relation:
\beq \label{2.9}
A_{n+1}(x)= A_{n}(x)+\Phi_n(x) A_{n}(x)
\eeq
where 
\beq \label{2.10}
\Phi_n(x)=-b_{n+1}\left( \begin{array}{cc} \psi^1_n(x) \psi^2_n(x)  &  \left( \psi^2_n(x) \right)^2 \\ 
-\left( \psi^1_{n}(x) \right)^2 & - \psi^1_n(x)\psi^2_n(x) \end{array} \right).
\eeq
By noting $\left( I+\Phi_n(x) \right)^{-1}=I-\Phi_n(x)$ we get also that 
\beq \label{2.10.1}
A_n(x)= A_{n+1}(x)-\Phi_n(x) A_{n+1}(x).
\eeq
By extending the definition of $M_I$, we assume also that $\left \| \Phi_n(x) \right \| \leq |b_{n+1}| M_I^2$ for any $n \in \bbN$, $x \in I$. 

Thus, we immediately see that along stretches where $b_{n+1}=0$, $A_n(x)$ is constant in $n$. 
Moreover, in this case, $A_{n}(x+\frac{a}{n})-A_{n}(x)$ is small for large $n$, so we can approximate $p_n(x+\frac{a}{n})$ by  
\mbox{$A_{n,1}(x) \psi^1_n(x+\frac{a}{n})+A_{n,2}(x) \psi^2_n(x+\frac{a}{n})$}. 

For constant $A$ we have

\begin{lemma}\label{lemma1}
For any $A=\left(\begin{array}{c} A_1 \\ A_2 \end{array} \right) \in \bbC^2$, let 
\beq \label{2.11.1}
\vphi^A_n(x)=A_1\psi^1_n(x)+A_2 \psi^2_n(x)
\eeq
and let 
\beq \no
K^A_n(x,y)=\sum_{j=0}^{n-1} \vphi^A_j(x) \vphi^A_j(y).
\eeq

Then, for any $x \in (-2,2)$ and any $a,b \in \bbC$,
\beq \label{2.12.0}
\lim_{n \rightarrow \infty}\frac{1}{\left (A_1^2+A_2^2-A_1A_2 x \right) } \frac{ K^A_n(x+\frac{a}{n},x+\frac{b}{n})}{n}= 
\frac{2\sin \left((\sqrt{4-x^2})^{-1} (b-a) \right)}{\sqrt{4-x^2} (b-a)}.
\eeq
Moreover, for any $C>0$ and any closed interval $I \subseteq (-2,2)$, the convergence is uniform in $1/C < \parallel A \parallel <C$, $|a|, |b|<C$ and $x \in I$.
\end{lemma}

\begin{proof}
It is a simple calculation, using \eqref{2.3} and \eqref{2.4}, to see that for $a \neq b$
\beq \no
\lim_{n \rightarrow \infty} \frac{ K^A_n(x+\frac{a}{n},x+\frac{b}{n})}{n}= 
\frac{2\left (A_1^2+A_2^2-A_1A_2x \right)\sin \left((\sqrt{4-x^2})^{-1} (b-a) \right)}{\sqrt{4-x^2} (b-a)}
\eeq
and that the convergence is uniform in $\parallel A \parallel <C$, $x \in I$, and $|a|,|b|<C$ with $|a-b|>\delta$ for any $\delta>0$. The same analyticity argument as the one given after \eqref{1.7} shows that convergence holds also for $a=b$ and is uniform in $\parallel A \parallel <C$, $x \in I$, and $|a|,|b|<C$. 

Since, for $x \in I$, $|A_1 A_2 x| < d (A_1^2+A_2^2)$ for some $d<1$, we get that after dividing by $\left (A_1^2+A_2^2-A_1A_2x \right)$ the convergence is still uniform for \mbox{$1/C < \parallel A \parallel$}.
\end{proof}

Now, let $\mu^{(\ell)}$ be the measure associated with the Jacobi coefficients 
\beq \label{2.13}
a_n \equiv 1 \quad b_n=\left \{ \begin{array}{cc} v_j & n= N_j, \quad j \leq \ell \\
0 & \textrm{otherwise.} \end{array} \right.
\eeq
Let $K^{(\ell)}_n$ be the CD  kernel, and $p_n^{(\ell)}$ the orthogonal polynomials associated with $\mu^{(\ell)}$, and let $A_n^{(\ell)}(x)$ be defined by 
\beq \label{2.13.1}
\begin{split}
p_n^{\ell}(x)&=A_{n,1}^{(\ell)}(x)\psi^1_n(x)+A_{n,2}^{(\ell)}(x) \psi^2_n(x) \\
p_{n-1}^{(\ell)}(x)&=A_{n,1}^{(\ell)}(x)\psi^1_{n-1}(x)+A_{n,2}^{(\ell)}(x) \psi^2_{n-1}(x).
\end{split}
\eeq

\begin{lemma} \label{lemma2}
For any $x \in (-2,2)$, $a, b \in \bbC$,
\beq \label{2.12}
\begin{split}
&\lim_{n \rightarrow \infty} \frac{ K_n^{(\ell)}(x+\frac{a}{n},x+\frac{b}{n})}{n\left (A_{n,1}^{(\ell)}(x)^2+A_{n,2}^{(\ell)}(x)^2-A_{n,1}^{(\ell)}(x)A_{n,2}^{(\ell)}(x) x \right) } \\
& \quad = \frac{2\sin \left((\sqrt{4-x^2})^{-1} (b-a) \right)}{\sqrt{4-x^2} (b-a)}
\end{split}
\eeq
where, again, $\frac{\sin(0)}{0}=1$.
Moreover, for any $m$ and $C>0$, the convergence is uniform in $x \in I_m$, and in $|a|,|b|\leq C$. That is, for any $m$, $C>0$ and any  $\varepsilon>0$ there exists $N(\varepsilon,m)$ so that for any $x \in I_m$, any $|a|,|b|\leq C$, and any $n \geq N(\varepsilon,m)$,
\beq \label{2.13.3}
\begin{split}
& \Bigg |\frac{ K_n^{(\ell)}(x+\frac{a}{n},x+\frac{b}{n})}{n\left (A_{n,1}^{(\ell)}(x)^2+A_{n,2}^{(\ell)}(x)^2-A_{n,1}^{(\ell)}(x)A_{n,2}^{(\ell)}(x) x \right) } \\
& \quad - \frac{2\sin \left((\sqrt{4-x^2})^{-1} (b-a) \right)}{\sqrt{4-x^2} (b-a)} \Bigg|
\end{split} <\varepsilon.
\eeq
\end{lemma}

\begin{proof}
As in the proof of Lemma \ref{lemma1}, it suffices to prove uniform convergence for $|a-b|>\delta$ for each $\delta>0$.
Since, for any $n \geq N_\ell+1$, $A_n^{(\ell)}(x)=A_{N_\ell+1}^{(\ell)}(x)$ we use, for simplicity of notation, $\widetilde{A}(x)=A_{N_\ell+1}^{(\ell)}(x)$.  
Fix  $x_0 \in I_m$ for some $m$ and let
\beq \no
\vphi_n(x)=\widetilde{A}_1(x_0)\psi^1_n(x)+\widetilde{A}_2(x_0) \psi^2_n(x),
\eeq
(note the presence of both $x_0$ and $x$ in the definition of $\vphi_n(x)$).
Let, for $x \neq y$,  
\beq \no
\widehat{K}_n(x_0;x,y)=\frac{ \vphi_n(x) \vphi_{n-1}(y)-\vphi_n(y)\vphi_{n-1}(x)}{x-y}.
\eeq 

By Lemma \ref{lemma1}
\beq \label{2.14}
\begin{split}
& \lim_{n \rightarrow \infty}\frac{1}{\left (\widetilde{A}_1(x_0)^2+\widetilde{A}_2(x_0)^2-\widetilde{A}_1(x_0)\widetilde{A}_2(x_0) x_0 \right) } \frac{ \widehat{K}_n(x_0;x_0+\frac{a}{n},x_0+\frac{b}{n})}{n} \\
& \quad = \frac{2\sin \left((\sqrt{4-x^2})^{-1} (b-a) \right)}{\sqrt{4-x^2} (b-a)}
\end{split}
\eeq
uniformly. 

It is clear that $f(z)=\widetilde{A}(z)$ is an analytic vector-valued function of $z$ (its components are polynomials in $z$). Thus, $\widetilde{A}(x_0+\frac{a}{n}) \rightarrow \widetilde{A}(x_0)$  as $n \rightarrow \infty$ uniformly for $a$ in compact subsets of the plane. Using this, and \eqref{1.5}, it is an easy (though somewhat tedious) computation to see that 
\beq \label{2.15}
\begin{split}
&\lim_{n \rightarrow \infty} 
\Bigg | \frac{1}{\left (\widetilde{A}_1(x_0)^2+\widetilde{A}_2(x_0)^2-\widetilde{A}_1(x_0)\widetilde{A}_2(x_0) x_0 \right) } \frac{ K_n^{(\ell)}(x_0+\frac{a}{n},x_0+\frac{b}{n})}{n} \\
&- \frac{1}{\left (\widetilde{A}_1(x_0)^2+\widetilde{A}_2(x_0)^2-\widetilde{A}_1(x_0)\widetilde{A}_2(x_0) x_0 \right) } 
\frac{\widehat{ K}_n(x_0;x_0+\frac{a}{n},x_0+\frac{b}{n})}{n} \Bigg | =0
\end{split}
\eeq
uniformly in $x_0 \in I_m$ and $|a|,|b| < C$, $|a-b|>\delta$. 

Combining \eqref{2.14} and \eqref{2.15} shows 
\beq \label{2.15.1}
\begin{split}
&\lim_{n \rightarrow \infty} \frac { K_n^{(\ell)}(x_0+\frac{a}{n},x_0+\frac{b}{n})}{n\left (\widetilde{A}_1(x_0)^2+\widetilde{A}_2(x_0)^2-\widetilde{A}_1(x_0)\widetilde{A}_2(x_0) x_0 \right) } \\
& \quad = \frac{2\sin \left((\sqrt{4-x^2})^{-1} (b-a) \right)}{\sqrt{4-x^2} (b-a)}
\end{split}
\eeq
uniformly. Noting that, for fixed $a$, the members of the sequence, as well as the limiting function, are all analytic in $b$, we obtain, as before, the limit for $a=b$. This ends the proof.
\end{proof}

\section{Proof of Theorem \ref{theorem}}

\begin{proof}[Proof of Theorem \ref{theorem}]
Given $\{v_n\}_{n=1}^\infty$ let $\{m_n\}_{n=1}^\infty$ satisfy $m_n \rightarrow \infty$ monotonically as $n \rightarrow \infty$ and $ v_n m_n^2 M_{m_n}^4 \rightarrow 0$ where we recall that $M_m \equiv M_{I_m}$ and $I_m=[-2+\frac{1}{m},2-\frac{1}{m}]$ (recall \eqref{2.8.1}). This is possible since $v_n \rightarrow 0$, so for any $r=1,2,\ldots$, there exists $N_r$ so that for any $n \geq N_r$, $|v_n|<\frac{1}{M_r^4 r^4}$. Thus, we can choose $m_1=m_2=\ldots=m_{N_2}=1$, $m_{N_2+1}=m_{N_2+2}=\ldots=m_{N_3}=2$ and generally, $m_{N_r+1}=m_{N_r+2}=\ldots=m_{N_{r+1}}=r$. In particular, $|v_n| M_{m_n}$ is bounded.

Assume we've fixed $\{N_j\}_{j=1}^\ell$. Let $\ti{I}_{\ell}=I_{m_{\ell+1}-1/\ell}$ (a closed interval contained in the interior of $I_{m_{\ell+1}}$), and consider $a,b \in \bbC$ with $|a|, |b| \leq \ell$, $\Im(a), \Im(b) \leq 1$. By Lemma \ref{lemma2}, there exists $\widehat{N}(\ell)$ such that for any $n \geq \widehat{N}(\ell)$,
\beq \label{2.16}
\begin{split}
& \Bigg |\frac{ K_n^{(\ell)}(x+\frac{a}{n},x+\frac{b}{n})}{n\left (A_{n,1}^{(\ell)}(x)^2+A_{n,2}^{(\ell)}(x)^2-A_{n,1}^{(\ell)}(x)A_{n,2}^{(\ell)}(x) x \right) } \\
& \quad - \frac{2\sin \left((\sqrt{4-x^2})^{-1} (b-a) \right)}{\sqrt{4-x^2} (b-a)} \Bigg| < \frac{1}{\ell}
\end{split}
\eeq
for any $x \in \ti{I}_{\ell}$ and any $|a|, |b| \leq \ell$. By taking $\widehat{N}(\ell)$ large enough, we may also assume that $\Re{\left(x+\frac{a}{n}\right)}, \Re{\left( x+\frac{b}{n}\right)} \in I_{m_{\ell+1}}$ for any $n \geq \widehat{N}(\ell)$. 
Finally, we may assume that 
\beq \label{2.16.1}
\frac{\left| A_{n,1}^{(\ell)}\left(x+\frac{a}{n}\right)\right|^2+\left|A_{n,2}^{(\ell)}\left(x+\frac{a}{n}\right) \right|^2}{\left|A_{n,1}^{(\ell)} \left(x\right) \right|^2+\left| A_{n,2}^{(\ell)} \left( x\right) \right|^2 } \leq 2
\eeq
and 
\beq \label{2.16.2}
\frac{\left| A_{n,1}^{(\ell)}\left(x+\frac{b}{n}\right)\right|^2+\left|A_{n,2}^{(\ell)}\left(x+\frac{b}{n}\right) \right|^2}{\left|A_{n,1}^{(\ell)} \left(x\right) \right|^2+\left| A_{n,2}^{(\ell)} \left( x\right) \right|^2 } \leq 2
\eeq
for $x \in \ti{I}_\ell$, $a,b \in \bbC$ with $|a|, |b| \leq \ell$ and $n \geq \widehat{N}(\ell)$. This is because $\left|A_{n,1}^{(\ell)} \left(x\right) \right|^2+\left| A_{n,2}^{(\ell)} \left( x\right) \right|^2$ is a continuous, non-vanishing function which is independent of $n$ for $n \geq N_{\ell}$ (recall \eqref{2.9}).

We will show that as long as we pick $N_{\ell+1} \geq \widehat{N}(\ell)$ (inductively), \eqref{1.12} holds uniformly for $a,b$ in compact subsets of $\bbR$ and $x \in$ closed subintervals of $(-2,2)$. Our strategy will be to first prove that 
\beq \label{2.17}
\begin{split}
& \Bigg |\frac{ K_n(x+\frac{a}{n},x+\frac{b}{n})}{n\left (A_{n,1}(x)^2+A_{n,2}(x)^2-A_{n,1}(x)A_{n,2}(x) x \right) } \\
& \quad - \frac{2\sin \left((\sqrt{4-x^2})^{-1} (b-a) \right)}{\sqrt{4-x^2} (b-a)} \Bigg| \rightarrow 0
\end{split}
\eeq
uniformly for complex $a,b$ with $|a-b|>\delta$ and $\Im{(a)}, \Im{(b)} \leq 1$ and deduce \eqref{1.12} using the analyticity argument used repeatedly above. 

Note that, for any $N_\ell \leq n < N_{\ell+1}$, $K_n=K_n^{(\ell)}$ and $A_n=A_n^{(\ell)}$. Thus, for any 
$N_{\ell+1} \leq n < N_{\ell+2}$
\beq \no
\begin{split}
& \frac{ K_n(x+\frac{a}{n},x+\frac{b}{n})}{n\left (A_{n,1}(x)^2+A_{n,2}(x)^2-A_{n,1}(x)A_{n,2}(x) x \right) }\\
&\quad -\frac{ K_n^{(\ell)}(x+\frac{a}{n},x+\frac{b}{n})}{n\left (A_{n,1}^{(\ell)}(x)^2+A_{n,2}^{(\ell)}(x)^2-A_{n,1}^{(\ell)}(x)A_{n,2}^{(\ell)}(x) x \right) } \\
&=\frac{ K_n^{(\ell+1)}(x+\frac{a}{n},x+\frac{b}{n})}{n\left (A_{n,1}^{(\ell+1)}(x)^2+A_{n,2}^{(\ell+1)}(x)^2-A_{n,1}^{(\ell+1)}(x)A_{n,2}^{(\ell+1)}(x) x \right) }\\
&\quad -\frac{ K_n^{(\ell)}(x+\frac{a}{n},x+\frac{b}{n})}{n\left (A_{n,1}^{(\ell)}(x)^2+A_{n,2}^{(\ell)}(x)^2-A_{n,1}^{(\ell)}(x)A_{n,2}^{(\ell)}(x) x \right) },
\end{split}
\eeq
and since \eqref{2.16} holds for any $n \geq N_{\ell+1}$, $x \in \ti{I}_{\ell}$ and $|a|, |b| < \ell$, it will be enough to show that 
\beq \label{2.18}
\begin{split}
&\max_{N_{\ell+1} \leq n <N_{\ell+2}, x \in \ti{I}_{\ell}} \Bigg|\frac{ K_n^{(\ell+1)}(x+\frac{a}{n},x+\frac{b}{n})}{n\left (A_{n,1}^{(\ell+1)}(x)^2+A_{n,2}^{(\ell+1)}(x)^2-A_{n,1}^{(\ell+1)}(x)A_{n,2}^{(\ell+1)}(x) x \right) }\\
&\quad -\frac{ K_n^{(\ell)}(x+\frac{a}{n},x+\frac{b}{n})}{n\left (A_{n,1}^{(\ell)}(x)^2+A_{n,2}^{(\ell)}(x)^2-A_{n,1}^{(\ell)}(x)A_{n,2}^{(\ell)}(x) x \right) } \Bigg|  \rightarrow 0
\end{split}
\eeq
as $\ell \rightarrow \infty$, uniformly in $|a|, |b| <C$, $|a-b|>\delta$.

For notational simplicity, let $\kappa_n^{(\ell)}(x)=A_{n,1}^{(\ell)}(x)^2+A_{n,2}^{(\ell)}(x)^2-A_{n,1}^{(\ell)}(x)A_{n,2}^{(\ell)}(x) x$, and write
\beq \label{decomposition}
\begin{split}
& \Bigg|\frac{ K_n^{(\ell+1)}(x+\frac{a}{n},x+\frac{b}{n})}{n \kappa_n^{(\ell+1)}(x) } -\frac{ K_n^{(\ell)}(x+\frac{a}{n},x+\frac{b}{n})}{n\kappa_n^{(\ell)}(x) } \Bigg| \\
& \leq \Bigg|\frac{ K_n^{(\ell+1)}(x+\frac{a}{n},x+\frac{b}{n})}{n \kappa_n^{(\ell+1)}(x) } -\frac{ K_n^{(\ell)}(x+\frac{a}{n},x+\frac{b}{n})}{n\kappa_n^{(\ell+1)}(x) } \Bigg| \\
&+\Bigg|\frac{ K_n^{(\ell)}(x+\frac{a}{n},x+\frac{b}{n})}{n \kappa_n^{(\ell+1)}(x) }-\frac{ K_n^{(\ell)}(x+\frac{a}{n},x+\frac{b}{n})}{n\kappa_n^{(\ell)}(x) } \Bigg| \\
&=\Bigg|\frac{ K_n^{(\ell+1)}(x+\frac{a}{n},x+\frac{b}{n}) - K_n^{(\ell)}(x+\frac{a}{n},x+\frac{b}{n})}{n\kappa_n^{(\ell+1)}(x) } \Bigg| \\
&+ \Bigg|\frac{ K_n^{(\ell)}(x+\frac{a}{n},x+\frac{b}{n})}{n\kappa_n^{(\ell)}(x) }  \Bigg| \Bigg|\frac{\kappa_n^{(\ell)}(x)-\kappa_n^{(\ell+1)}(x)}{ \kappa_n^{(\ell+1)}(x) }\Bigg|.
\end{split}
\eeq

We treat the summands on the right hand side one by one, starting with $\Bigg|\frac{ K_n^{(\ell+1)}(x+\frac{a}{n},x+\frac{b}{n}) - K_n^{(\ell)}(x+\frac{a}{n},x+\frac{b}{n})}{n\kappa_n^{(\ell+1)}(x) } \Bigg| $.
Note that for any $n < N_{\ell+1}$, $p_n^{(\ell)}(x)=p_n^{(\ell+1)}(x)$. For $n=N_{\ell+1}$, by \eqref{1.2},
\beq \no
\begin{split}
p^{(\ell+1)}_n(x)&=xp^{(\ell+1)}_{n-1}(x)-p_{n-2}^{(\ell+1)}(x)-v_{\ell+1}p^{(\ell+1)}_{n-1}(x) \\
&=xp^{(\ell)}_{n-1}(x)-p^{(\ell)}_{n-2}(x)-v_{\ell+1}p^{(\ell)}_{n-1}(x) =p^{(\ell)}_{n}(x)-v_{\ell+1}p^{(\ell)}_{n-1}(x).
\end{split}
\eeq
Thus, for any $N_{\ell+1}\leq n < N_{\ell+2}$,
\beq \label{2.19.1}
\begin{split}
\left(\begin{array}{c} p^{(\ell+1)}_n(x) \\ p^{(\ell+1)}_{n-1}(x) \end{array} \right) &=T_{n-N_{\ell+1}}(x) \left(\begin{array}{c} p^{(\ell+1)}_{N_{\ell+1}}(x) \\ p^{(\ell+1)}_{N_{\ell+1}-1}(x) \end{array} \right) \\
&=T_{n-N_{\ell+1}}(x) \left(\begin{array}{c} p^{(\ell)}_{N_{\ell+1}}(x)-v_{\ell+1}p^{(\ell)}_{N_{\ell+1}-1}(x) \\ p^{(\ell)}_{N_{\ell+1}-1}(x) \end{array} \right) \\
&=T_{n-N_{\ell+1}}(x) \left(\begin{array}{c} p^{(\ell)}_{N_{\ell+1}}(x)\\ p^{(\ell)}_{N_{\ell+1}-1}(x) \end{array} \right) \\
&\quad -v_{\ell+1}p^{(\ell)}_{N_{\ell+1}-1}(x)T_{n-N_{\ell+1}}(x)  \left(\begin{array}{c} 1\\ 0 \end{array} \right), \\
\end{split}
\eeq
which implies
\beq \label{2.19}
\begin{split}
\left(\begin{array}{c} p^{(\ell+1)}_n(x) \\ p^{(\ell+1)}_{n-1}(x) \end{array} \right) 
&=  \left(\begin{array}{c} p^{(\ell)}_n(x) \\ p^{(\ell)}_{n-1}(x) \end{array} \right) -v_{\ell+1}p^{(\ell)}_{N_{\ell+1}-1}(x)T_{n-N_{\ell+1}}(x)  \left(\begin{array}{c} 1\\ 0 \end{array} \right) \\
&=  \left(\begin{array}{c} p^{(\ell)}_n(x) \\ p^{(\ell)}_{n-1}(x) \end{array} \right)
-v_{\ell+1}p^{(\ell)}_{N_{\ell+1}-1}(x)  \left(\begin{array}{c} \psi_{n-N_{\ell+1}}^1(x)\\  \psi_{n-1-N_{\ell+1}}^1(x) \end{array} \right).
\end{split}
\eeq
Using Lemma \ref{bound-extension} and $2|\alpha||\beta| \leq |\alpha|^2 +|\beta|^2$, it follows that, for $N_{\ell+1}\leq n < N_{\ell+2}$, $x \in \ti{I}_\ell$ and $a, b$ with $\Im(a), \Im(b) \leq 1$, and $|a|, |b| \leq \ell$,
\beq \label{2.20}
\begin{split}
&\Bigg|\left(p^{(\ell+1)}_n\left(x+\frac{a}{n} \right)p^{(\ell+1)}_{n-1}\left(x+\frac{b}{n}\right)-p^{(\ell+1)}_n\left(x+\frac{b}{n}\right)p^{(\ell+1)}_{n-1}\left(x+\frac{a}{n}\right) \right)\\
&-\left( p^{(\ell)}_n\left(x+\frac{a}{n}\right)p^{(\ell)}_{n-1}\left(x+\frac{b}{n}\right)-p^{(\ell)}_n\left(x+\frac{b}{n}\right)p^{(\ell)}_{n-1}\left(x+\frac{a}{n}\right) \right)\Bigg| \\
& \leq 4|v_{\ell+1}|^2 M_{I_{m_{\ell+1}}}^2\left( \left|p^{(\ell)}_{N_{\ell+1}-1}\left(x+\frac{a}{n}\right) \right|^2+
\left|p^{(\ell)}_{N_{\ell+1}-1}\left(x+\frac{b}{n}\right)\right|^2 \right)\\
&+4|v_{\ell+1}| M_{I_{m_{\ell+1}}} \Bigg (\left|p^{(\ell)}_{N_{\ell+1}-1}\left(x+\frac{a}{n}\right) \right|^2+
\left|p^{(\ell)}_{N_{\ell+1}-1}\left(x+\frac{b}{n}\right)\right|^2 \\
&+\left|p^{(\ell)}_{n}\left(x+\frac{a}{n}\right) \right|^2+\left|p^{(\ell)}_{n-1}\left(x+\frac{a}{n}\right) \right|^2
+\left|p^{(\ell)}_{n}\left(x+\frac{b}{n}\right) \right|^2 \\
&+\left|p^{(\ell)}_{n-1}\left(x+\frac{b}{n}\right) \right|^2\Bigg).
\end{split}
\eeq

Now, by \eqref{2.13.1} and Lemma \ref{bound-extension}, for any $x \in I$, $I \subseteq (-2,2)$ closed and any $-1 \leq t \leq 1$,
\beq \label{2.21}
\begin{split}
& \left |p_n^{(\ell)}\left(x+\frac{it}{n}\right)\right|^2+\left|p_{n+1}^{(\ell)}\left(x+\frac{it}{n}\right)\right|^2 \\
& \leq M_{I}^2 \left (\left|A_{n,1}^{(\ell)}\left(x+\frac{it}{n}\right)\right|^2+\left|A_{n,2}^{(\ell)}\left(x+\frac{it}{n}\right)\right|^2 \right).
\end{split}
\eeq

Moreover, by \eqref{2.10.1}, for any $n \geq N_{\ell}$,
\beq \label{2.22}
\left (\left|A_{n,1}^{(\ell)}(x)\right|^2+\left|A_{n,2}^{(\ell)}(x)\right|^2 \right) \leq \left(1+|v_{\ell+1}| M_I \right)^2 \left (\left|A_{n,1}^{(\ell+1)}(x)\right|^2+\left|A_{n,2}^{(\ell+1)}(x)\right|^2 \right).
\eeq

Combining \eqref{2.21}, \eqref{2.16.1}, \eqref{2.16.2} and \eqref{2.22} with \eqref{2.20} we get (recall  \mbox{$M_{I_{m_{\ell+1}}} \geq 1$}) 
\beq \label{2.23}
\begin{split}
&\Bigg|\left(p^{(\ell+1)}_n\left(x+\frac{a}{n} \right)p^{(\ell+1)}_{n-1}\left(x+\frac{b}{n}\right)-p^{(\ell+1)}_n\left(x+\frac{b}{n}\right)p^{(\ell+1)}_{n-1}\left(x+\frac{a}{n}\right) \right)\\
&-\left( p^{(\ell)}_n\left(x+\frac{a}{n}\right)p^{(\ell)}_{n-1}\left(x+\frac{b}{n}\right)-p^{(\ell)}_n\left(x+\frac{b}{n}\right)p^{(\ell)}_{n-1}\left(x+\frac{a}{n}\right) \right)\Bigg| \\
& \leq\left( 8|v_{\ell+1}|^2 M_{I_{m_{\ell+1}}}^4+24|v_{\ell+1}| M_{I_{m_{\ell+1}}}^4 \right)\left(1+|v_{\ell+1}|M_{I_{m_{\ell+1}}} \right)^2 \\
& \times \left( \left|A_{n,1}^{(\ell+1)}(x) \right|^2 + \left|A_{n,2}^{(\ell+1)}(x) \right|^2\right)\\
& \leq  D |v_{\ell+1}|M_{I_{m_{\ell+1}}}^4 \left( \left|A_{n,1}^{(\ell+1)}(x) \right|^2 + \left|A_{n,2}^{(\ell+1)}(x) \right|^2\right)\\
\end{split}
\eeq
where $D$ is some constant which is independent of $n$, $\ell$ and $x$.

On the other hand, it is easy to see that 
\beq \label{2.24}
\left|\kappa_n^{(\ell+1)}(x) \right| \geq \left(1-\frac{|x|}{2} \right)\left( \left|A_{n,1}^{(\ell+1)}(x) \right|^2+\left|A_{n,2}^{(\ell+1)} \right|^2 \right).
\eeq 
It follows from \eqref{2.23} and \eqref{2.24} that 
\beq \label{2.25}
\Bigg|\frac{ K_n^{(\ell+1)}(x+\frac{a}{n},x+\frac{b}{n}) - K_n^{(\ell)}(x+\frac{a}{n},x+\frac{b}{n})}{n\kappa_n^{(\ell+1)}(x) } \Bigg| \leq \frac{2 D |v_{\ell+1}|M_{I_{m_{\ell+1}}}^4 }{(2-|x|) |b-a|}
\eeq
which implies, by the choice of $m_n$, that
\beq \label{2.26}
 \max_{N_{\ell+1} \leq n <N_{\ell+2}, x \in \ti{I}_{\ell}}\Bigg|\frac{ K_n^{(\ell+1)}(x+\frac{a}{n},x+\frac{b}{n}) - K_n^{(\ell)}(x+\frac{a}{n},x+\frac{b}{n})}{n\kappa_n^{(\ell+1)}(x) } \Bigg| \rightarrow 0
\eeq
as $\ell \rightarrow \infty$, uniformly for $a,b$ with $|a|, |b| \leq C$ and $\Im(a), \Im(b) \leq 1$ and satisfying $|b-a|>\delta$. 

As for the second term on the right hand side of \eqref{decomposition}, we write
\beq \label{2.27}
\begin{split}
& \left(\left(A_{n,1}^{(\ell+1)}(x)\right)^2+\left(A_{n,2}^{(\ell+1)}(x)\right)^2-A_{n,1}^{(\ell+1)}(x)A_{n,2}^{(\ell+1)}x \right) \\
&- \left(\left(A_{n,1}^{(\ell)}(x)\right)^2+\left(A_{n,2}^{(\ell)}(x)\right)^2-A_{n,1}^{(\ell)}(x)A_{n,2}^{(\ell)}x\right) \\
& =\left(A_{n,1}^{(\ell+1)}(x)-A_{n,1}^{(\ell)}(x) \right) \left(A_{n,1}^{(\ell+1)}(x)+A_{n,1}^{(\ell)}(x)  \right) \\
& +\left(A_{n,2}^{(\ell+1)}(x)-A_{n,2}^{(\ell)}(x) \right) \left(A_{n,2}^{(\ell+1)}(x)+A_{n,2}^{(\ell)}(x)  \right) \\
&-xA_{n,1}^{(\ell+1)}(x) \left(A_{n,2}^{(\ell+1)}(x)-A_{n,2}^{(\ell)}(x) \right)-xA_{n,2}^{(\ell)}(x) \left(A_{n,1}^{(\ell+1)}(x)-A_{n,1}^{(\ell)}(x) \right),
\end{split}
\eeq
which implies, by \eqref{2.9} and \eqref{2.10.1},
\beq \no
\begin{split}
& \left| \kappa_n^{(\ell+1)}(x)-\kappa_{n}^{(\ell)}(x) \right| \\
&= \Bigg| \left(\left(A_{n,1}^{(\ell+1)}(x)\right)^2+\left(A_{n,2}^{(\ell+1)}(x)\right)^2-A_{n,1}^{(\ell+1)}(x)A_{n,2}^{(\ell+1)}x \right) \\
&- \left(\left(A_{n,1}^{(\ell)}(x)\right)^2+\left(A_{n,2}^{(\ell)}(x)\right)^2-A_{n,1}^{(\ell)}(x)A_{n,2}^{(\ell)}x\right) \Bigg| \\
& \leq |v_{\ell+1}|M_{I_{m_{\ell+1}}}^2  \Bigg(\left|A_{n,1}^{(\ell+1)}(x)\right|^2+\left|A_{n,1}^{(\ell+1)}(x)A_{n,1}^{(\ell)}(x)\right| \\
&+ \left|A_{n,2}^{(\ell+1)}(x)\right|^2+\left|A_{n,2}^{(\ell+1)}A_{n,2}^{(\ell)}(x)\right| \\
&+|x|\left|A_{n,1}^{(\ell+1)}(x)A_{n,2}^{(\ell+1)}(x) \right|+|x|\left|A_{n,1}^{(\ell)}(x)A_{n,2}^{(\ell)}(x) \right|\Bigg), \\
\end{split}
\eeq
so that
\beq \label{2.28}
\begin{split}
& \left| \kappa_n^{(\ell+1)}(x)-\kappa_{n}^{(\ell)}(x) \right| \\
& \leq |v_{\ell+1}|M_{I_{m_{\ell+1}}}^2  (2+|x|) \Bigg(\left|A_{n,1}^{(\ell+1)}(x)\right|^2+\left|A_{n,1}^{(\ell+2)}(x)\right|^2 + \left|A_{n,1}^{(\ell)}(x)\right|^2 \\
&+\left|A_{n,1}^{(\ell)}(x)\right|^2 \Bigg) \\
& \leq |v_{\ell+1}|M_{I_{m_{\ell+1}}}^2 \left (4+(1+|v_{\ell+1}|M_{I_{m_{\ell+1}}})^2\right)  \left(\left|A_{n,1}^{(\ell+1)}(x)\right|^2+\left|A_{n,1}^{(\ell+2)}(x)\right|^2 \right), \\
\end{split}
\eeq 
by \eqref{2.22} (recall $|x| \leq 2$).

Using again \eqref{2.24}, we deduce that 
\beq \label{2.29}
\begin{split}
\Bigg|\frac{\kappa_n^{(\ell)}(x)-\kappa_n^{(\ell+1)}(x)}{ \kappa_n^{(\ell+1)}(x) }\Bigg| \leq \frac{\widetilde{D} |v_{\ell+1}|M_{I_{m_{\ell+1}}}^2}{(2-|x|)},
\end{split}
\eeq
where $\widetilde{D}$ is some constant that is independent of $x$ and $n$. Since 
\beq \no
 \Bigg|\frac{ K_n^{(\ell)}(x+\frac{a}{n},x+\frac{b}{n})}{n\kappa_n^{(\ell)}(x) }  \Bigg| \leq \frac{\widetilde{C}}{(2-|x|)}
\eeq
for any $n \geq N_{\ell+1}$, with $\widetilde{C}$ some universal constant (recall \eqref{2.16}), we see that 
\beq \label{2.30}
\begin{split}
& \max_{N_{\ell+1} \leq n <N_{\ell+2}, x \in \ti{I}_{\ell}} 
 \Bigg|\frac{ K_n^{(\ell)}(x+\frac{a}{n},x+\frac{b}{n})}{n\kappa_n^{(\ell)}(x) }  \Bigg| \Bigg|\frac{\kappa_n^{(\ell)}(x)-\kappa_n^{(\ell+1)}(x)}{ \kappa_n^{(\ell+1)}(x) }\Bigg| \\
& \leq \frac{\widetilde{C} \widetilde{D} |v_{\ell+1}|M_{I_{m_{\ell+1}}}^2}{(2-|x|)^2}.
\end{split}
\eeq
This implies, by the choice of $m_n$, that 
\beq \label{2.31}
\max_{N_{\ell+1} \leq n <N_{\ell+2}, x \in \ti{I}_{\ell}} 
 \Bigg|\frac{ K_n^{(\ell)}(x+\frac{a}{n},x+\frac{b}{n})}{n\kappa_n^{(\ell)}(x) }  \Bigg| \Bigg|\frac{\kappa_n^{(\ell)}(x)-\kappa_n^{(\ell+1)}(x)}{ \kappa_n^{(\ell+1)}(x) }\Bigg| \rightarrow 0
\eeq
as $\ell \rightarrow \infty$.

Combining \eqref{decomposition}, \eqref{2.26} and \eqref{2.31} we obtain \eqref{2.18}, for $a \neq b$, uniformly for $x \in$ compact subsets of $(-2,2)$, $|a|, |b| \leq C$ with $\Im(a), \Im(b) \leq 1$, and $|a-b| > \delta$. This, in turn, implies \eqref{2.17} under the same conditions.

Now, having obtained the limit for $a \neq b$ we note, as before, that, for fixed $a$, the limiting function, as well as the members of the sequence, is analytic in a strip and has an analytic extension to $b=a$. By integrating along a closed path around $a$ (recall the convergence is uniform for $|b-a|> \delta$) we see there's convergence at $b=a$ in the interior of the strip as well. Taking $a=b=0$ in \eqref{2.17} we obtain that 
\beq \no
\lim_{n \rightarrow \infty} \frac{ K_n(x,x)}{n\left (A_{n,1}(x)^2+A_{n,2}(x)^2-A_{n,1}(x)A_{n,2}(x) x \right) } = \frac{2}{4-x^2}
\eeq 
which implies immediately
\beq \no
\lim_{n \rightarrow \infty }\frac{ K_n(x+\frac{a}{n},x+\frac{b}{n})}{K_n(x,x) }=\frac{\sin \left((\sqrt{4-x^2})^{-1} (b-a) \right)}{\sqrt{4-x^2}^{-1} (b-a)}
\eeq
for any $a,b$ in the interior of the strip of width $1$ around $\bbR$. In particular, this holds for $a, b \in \bbR$.
This concludes the proof of the theorem.
\end{proof}


\end{document}